\documentclass[12pt]{amsart}

\usepackage{epic,eepic,epsfig,amssymb,amsmath,amsthm,graphics,enumerate}

\usepackage{multimedia}

\usepackage{pdftricks}

\usepackage{pstricks-add}
\usepackage{filecontents}
\usepackage{graphicx}
\usepackage{pdftricks}
\usepackage{textcomp}
\usepackage{variations,multido}
\usepackage{pstricks,pstricks-add,pst-plot,pst-eucl,xkeyval}
\usepackage{xcolor,epic,eepic,multicol}
\begin{psinputs}
\usepackage{pstricks}
\usepackage{pst-fill,pst-grad,pst-plot}
\end{psinputs}

\usepackage{vmargin}
\usepackage{a4wide}
\usepackage{supertabular}
\usepackage{array}
\usepackage{multicol}

\setmarginsrb{3cm}{2cm}{3cm}{3cm}{75pt}{20pt}{20pt}{30mm}
\def\R{\mathbb R}
\def\N{\mathbb N}

\def\P{\mathbb P}

\def\J{\mathbb J}

\def\S{\mathbb S}

\newtheorem{theorem}{Theorem}[section]

\newtheorem{proposition}[theorem]{Proposition}

\theoremstyle{definition}

\newtheorem{remark}[theorem]{Remark}

\def\signal{\bigskip\begin{center} {\sc Ayadi Lazrag\par\vspace{3mm}
Univ. Nice Sophia Antipolis,\\ CNRS, LJAD, UMR 7351\\ 06100 Nice\\ FRANCE\par\vspace{3mm}
email:} \tt{ayadi.lazrag@gmail.com}\end{center}}

\begin{document}

\title[A geometric control proof of Franks' lemma for Geodesic Flows]{A geometric control proof of linear Franks' lemma for Geodesic Flows}

\author{A. Lazrag}

\begin{abstract}
We provide an elementary proof of the Franks lemma for geodesic flows that uses basic tools of geometric control theory.
\end{abstract}

\maketitle

\section{Introduction}

In 1971, John Franks stated and proved an elegant lemma (see \cite[lemma 1.1]{franks71}) showing how to perturb the derivative of a diffeomorphism along a periodic orbit by small perturbations of the diffeomorphism on a neighbourhood of the orbit. Since the original Franks' lemma concerns  diffeomorphisms, its proof is quite simple. The Franks lemma has since been proven in other interesting contexts such as geodesic flows (see \cite{cp02} and \cite{contreras10}) and more generally Hamiltonians flows (see \cite{vivier05}). In this work, we focus on the Franks lemma for geodesic flows. This problem was first studied in the particular case of surfaces by Contreras and Paternain (see \cite[Theorem 4.1] {cp02}). They proved that on any surface, the linearized Poincar\'e map along any geodesic segment of length $1$ can be freely perturbed in a neighborhood inside $\mbox{Sp}(1)$ by a $C^2$-small perturbation of the metric, where for every $m \in \N^*$, the symplectic group $\mbox{Sp}(m)$ is defined by 
$$
\mbox{Sp}(m):= \Bigl\{A \in M_{2m}(\R) \, \vert \,  A^*\J A=\J \Bigr\},  
$$
with 
$$
\J= \left[ \begin{matrix}
0&I_m\\
-I_m&0\\
\end{matrix} \right].
$$  
In 2010, Contreras studied the higher-dimensional analogue (see \cite[Theorem 7.1]{contreras10}). He generalized the previous result for a special set of metrics: those such that every geodesic segment of length $\frac{1}{2}$ has a point whose curvature matrix has all its eigenvalues distinct and separated by a uniform bound. The proof was long and technical.\\

Let $M$ be a closed manifold of dimension $n \geq 2$ endowed with a Riemannian metric $g$ and $S^gM$ be the unit tangent bundle. Given a geodesic arc of length $1$
$$
\gamma : \left[0,1\right] \longrightarrow S^gM,
$$
with unit speed and $\Sigma_0$ and  $\Sigma_1$ transverse sections at $\gamma(0)$ and $\gamma(1)$ respectively. Let $\P_g(\Sigma_0,\Sigma_1,\gamma)$ be a Poincar\'e map going from $\Sigma_0$ to $\Sigma_1$. One can choose $\Sigma_0$ and $\Sigma_1$ such that the \textit{linearized Poincar\'e} map
$$
P_g(\gamma)(t):=d_{\gamma(0)}\P_g(\Sigma_0,\Sigma_1,\gamma)
$$
is a symplectic endomorphism of $\mathbb{R}^{n-1} \times (\mathbb{R}^{n-1})^*$ (in local coordinates).
Let $\mathcal{R}^k(M)$, $k \in \N \cup \left\{+\infty \right\}$ be the set of all $C^k$ Riemannian metrics $g$ on $M$. If $n \geq 3$, we denote by $\mathcal{G}_1$ the set of Riemannian metrics on $M$ such that every unit geodesic segment of lenght $1$ admits a point where the curvature matrix has distinct eigenvalues. Denote by $\mathcal{R}^k(M,\mathcal{G}_1)$ the set of all Riemannian metrics $g$ on $M$ such that if $n=2$, $g \in \mathcal{R}^k(M)$ and for $n \geq 3$, $g \in \mathcal{R}^k(M) \cap  \mathcal{G}_1$. For every $k \geq 2$, $\mathcal{R}^k(M,\mathcal{G}_1)$ is an open and dense subset of $\mathcal{R}^k(M)$. Consider the map $S$ : $\mathcal{R}^k(M,\mathcal{G}_1) \longrightarrow \mbox{Sp}(n-1)$ given by $S(\bar{g})=P_{\bar{g}}(\gamma)(1)$. The following theorem summarizes the Franks lemma for geodesic flows on surfaces and its higher-dimensional analogue (under the Contreras assumption on the spectrum of the curvature matrix) with estimates on the size of perturbation in terms of the radius of the ball of $\mbox{Sp}(n-1)$.
\begin{theorem}
Let $g_0 \in \mathcal{R}^k(M,\mathcal{G}_1), 2 \leq k \leq \infty$. There exists $\bar{r},K > 0$ such that for any geodesic arc $\gamma$ of $g_0$ of lenght $1$ and any $r \in (0,\bar{r})$,
$$
B\Bigl(S(g_0),Kr\Bigr) \cap \mbox{Sp}(n-1) \subset S\Bigl(B_{C^k}(g_0,r)\Bigr).
$$
\end{theorem}
Let $\mathcal{F}:=\left\{\xi_1,...,\xi_N \right\}$ be a finite set of geodesic segments that are transverse to $\gamma$. We have the following result.
\begin{proposition}
For any tubular neighborhood $W$ of $\gamma$ and any finite set $\mathcal{F}$ of transverse geodesics, the support of the perturbation can be contained in $W \setminus V$ for some neighborhood $V$ of the transverse geodesics $\mathcal{F}$.
\end{proposition}
Franks' Lemma type results has many interesting applications. For instance, in \cite{cp02} Contreras and Paternain used it to show that the set of $C^{\infty}$ Riemannian metrics on $\S^2$ or $\R\P^2$ whose geodesic flow has positive topological entropy is open and dense in the $C^2$ topology. In \cite{contreras10} Contreras used the Franks lemma to prove that a $C^2$ generic Riemannian metric has a non-trivial hyperbolic basic set in its geodesic flow. The author says that this perturbation lemma is "the main technical difficulty of the paper".
Recently, Visscher (see \cite{visscher13}) gave a shorter and less technical proof for the two cases.\\

The purpose of the present paper is to provide a simple proof of the Franks lemma using geometric control tools. Such techniques have been initially introduced by Rifford and Ruggiero in \cite{rr12}. We mention that recently in a joint work with Rifford and Ruggiero, we obtained a Franks lemma at lower order $(r < K \sqrt{\delta}$) without the Contreras assumption (see \cite{lrr13}). \\

The paper is organized as follows. In the next section, we introduce some preliminaries in geometric control theory. We describe the relationship between local controllability and the properties of the End-Point mapping. In Section 3, we provide the proof of Theorem 1.1. Then, in Section 4, we provide the proof of Proposition 1.2.\\

\textbf{Acknowledgments:}
The author is very grateful to Ludovic Rifford for his suggestions, comments and careful reading of the paper. Special thanks to Lanouar Lazrag for his interesting remarks.

\section{Preliminaries in geometric control theory}\label{prelcontrol}

Our aim here is to provide sufficient conditions for first order local controllability results. This kind of results could be developed for nonlinear control systems on smooth manifolds.
For sake of simplicity, we restrict our attention here to the case of affine control systems on the set of (symplectic) matrices. We refer the interested reader to \cite{as04,coron07,riffordbook} for a further study in control theory.
\subsection{The End-Point mapping}
Let  us a consider a \textit{bilinear control system} on $M_{2m}(\R)$ (with $m, k \geq 1$), of the form
\begin{eqnarray}
\label{syscontrol}
\dot{X}(t)  = A(t) X(t) +  \sum_{i=1}^k u_i(t) B_i  X(t), \qquad \mbox{for a.e.} \quad  t,
\end{eqnarray}
where the \textit{state} $X(t)$ belongs to $M_{2m}(\R)$, the \textit{control} $u(t)$ belongs to $\R^k$, $t \in [0,T] \mapsto A(t)$ (with $T>0$) is a smooth map valued in $M_{2m}(\R)$,
and $B_1, \ldots, B_k$ are $k$ matrices in $M_{2m}(\R)$. Given $\bar{X}\in M_{2m}(\R)$ and $\bar{u} \in L^2 \bigl([0,T]; \R^k\bigr)$, the Cauchy problem
\begin{eqnarray}
\label{cauchyX}
\dot{X} (t) = A(t) X (t) + \sum_{i=1}^k \bar{u}_i(t) B_i X(t), \quad \mbox{for a.e.} \quad  t \in [0,T], \qquad X(0) = \bar{X},
\end{eqnarray}
possesses a unique solution $X_{\bar{X},\bar{u}}(\cdot)$. The \textit{End-Point mapping} associated with $\bar{X}$ in time $T>0$ is defined as
$$
\begin{array}{rcl}
E^{\bar{X},T} \, : \, L^2 \bigl([0,T]; \R^k\bigr)  & \longrightarrow & M_{2m}(\R) \\
u & \longmapsto & X_{\bar{X},u}(T).
\end{array}
$$
It is a smooth mapping. Given $\bar{X}\in M_{2m}(\R)$, $\bar{u} \in L^2 \bigl([0,T]; \R^k\bigr)$, and setting $\bar{X} (\cdot) :=  X_{\bar{X},\bar{u}}(\cdot)$, the differential of $E^{\bar{X},T} $ at $\bar{u}$ is given by the
linear operator
$$
\begin{array}{rccc}
D_{\bar{u}}E^{\bar{X},T}  \, : \, &L^2\bigl([0,T]; \R^k\bigr) & \longrightarrow & M_{2m}(\R) \\
&v & \longmapsto & Y(T),
\end{array}
$$
where $Y(\cdot)$ is the unique solution to the Cauchy problem
\begin{eqnarray}
\label{linearized}
\left\{
\begin{array}{l}
\dot{Y}(t) = A(t) Y(t) + \sum_{i=1}^k v_i(t) B_i \bar{X}(t) \quad \mbox{for a.e.} \quad t \in [0,T],\\
Y(0)= 0.
\end{array}
\right.
\end{eqnarray}
Note that if we denote by $S(\cdot)$ the solution to the Cauchy problem
\begin{equation}
\label{eq:St}
\left\{
\begin{array}{l}
\dot{S}(t) = A(t) S(t),\\
S(0)=I_{2m},
\end{array}
\right.
\end{equation}
then there holds
\begin{eqnarray}\label{raab}
 D_{\bar{u}}E^{\bar{X},T}  (v) =  \sum_{i=1}^k S(T) \int_0^{T} v_i(t) S(t)^{-1} B_i \bar{X}(t) \,dt,
\end{eqnarray}
for every $v \in L^2([0,T]; \R^k)$. \\

Let $\mbox{Sp}(m)$ be the symplectic group in $M_{2m}(\R)$ ($m\geq 1$), that is the smooth submanifold of matrices $X \in M_{2m}(\R)$ satisfying

$$
X^{*} \J X = \J \quad  \mbox{ where } \J = \left[ \begin{matrix} 0 & I_m \\ -I_m & 0 \end{matrix} \right].
$$
$\mbox{Sp}(m)$ has dimension $p:=2m(2m+1)/2$. Denote by $\mathcal{S}(2m)$ the set of symmetric matrices in $M_{2m}(\R)$. The tangent space to  $\mbox{Sp}(m)$ at the identity matrix is given by
$$
T_{I_{2m}} \mbox{Sp}(m) = \Bigl\{ Y \in M_{2m}(\R) \, \vert \, \J Y \in \mathcal{S}(2m)  \Bigr\}.
$$
Therefore, if there holds
\begin{eqnarray}\label{condsymp}
\J A (t), \, \J B_1, \, \ldots, \, \J B_k \in \mathcal{S}(2m) \qquad \forall t \in [0,T],
\end{eqnarray}
then $\mbox{Sp}(m)$ is invariant with respect to (\ref{syscontrol}), that is for every $\bar{X}\in \mbox{Sp}(m)$ and $\bar{u} \in L^2 \bigl([0,T]; \R^k\bigr)$,
$$
X_{\bar{X},u}(t) \in \mbox{Sp}(m) \qquad \forall t \in [0,T].
$$
In particular, this means that for every $\bar{X} \in \mbox{Sp}(m)$, the End-Point mapping $E^{\bar{X},T}$ is valued in $\mbox{Sp}(m)$. Given $\bar{X}\in \mbox{Sp}(m)$ and $\bar{u} \in L^2 \bigl([0,T]; \R^k\bigr)$,
we are interested in local controllability properties of (\ref{syscontrol}) around $\bar{u}$. The control system  (\ref{syscontrol}) is called \textit{controllable around} $\bar{u}$ in $\mbox{Sp}(m)$ (in time $T$) if
for every final state $X\in \mbox{Sp}(m)$ close to $X_{\bar{X},u}(T)$ there is a control $u\in L^2 \bigl([0,T]; \R^k\bigr)$ which steers $\bar{X}$ to $X$, that is such that $E^{\bar{X},T}(u)=X$.
Such a property is satisfied as soon as $E^{\bar{X},T}$ is locally open at $\bar{u}$.

\subsection{First order controllability results}

Given $T>0$, $\bar{X}\in \mbox{Sp}(m)$, a mapping $t\in [0,T] \mapsto A(t) \in M_{2m}(\R)$ and $k$ matrices  $B_1, \ldots, B_k \in M_{2m}(\R)$ satisfying (\ref{condsymp}), and $\bar{u} \in L^2 \bigl([0,T]; \R^k\bigr)$,
we say that the control system (\ref{syscontrol}) is  \textit{controllable at first order around} $\bar{u}$ in $\mbox{Sp}(m)$ if the mapping $E^{\bar{X},T} : L^2 \bigl([0,T]; \R^k\bigr) \rightarrow \mbox{Sp}(m)$ is a
\textit{submersion } at $\bar{u}$, that is if the linear operator
$$
D_{\bar{u}}E^{\bar{X},T} \, : \, L^2\bigl([0,T]; \R^k\bigr)  \longrightarrow  T_{\bar{X}(T)} \mbox{Sp}(m),
$$
is surjective (with $\bar{X}(T):=X_{\bar{X},u}(T)$). The following sufficient condition for first order controllability is given in \cite[Proposition 2.1]{rr12}. For sake of completeness, we provide its proof.

\begin{proposition}
\label{LIEPROP1}
Let $T>0$, $t\in [0,T] \mapsto A(t)$ a smooth mapping and $B_1, \ldots, B_k \in M_{2m}(\R)$ be matrices in $M_{2m}(\R)$ satisfying (\ref{condsymp}). Define the $k$ sequences of smooth mappings
$$
\{B_1^j\}, \ldots, \{B_k^j\} : [0,T] \rightarrow  T_{I_{2m}} \mbox{Sp}(m)
$$
by
\begin{eqnarray}
\label{brackets}
\left\{
\begin{array}{l}
B_i^0(t) = B_i\\
B_i^j(t) = \dot{B}_{i}^{j-1}(t) + B_i^{j-1}(t) A(t)  - A(t) B_i^{j-1}(t),
\end{array}
\right.
\end{eqnarray}
for every $t\in [0,T]$ and every $i \in \{1, \ldots, k\}$. Assume that there exists some $\bar{t} \in [0,T]$ such that
\begin{eqnarray}\label{conditionLIEPROP}
\mbox{Span} \Bigl\{ B_i^j(\bar{t}) \, \vert \, i \in \{1,\ldots, k\}, j\in \N\Bigr\} = T_{I_{2m}} \mbox{Sp}(m).
\end{eqnarray}
Then for every $\bar{X} \in \mbox{Sp}(m)$, the control system (\ref{syscontrol}) is controllable at first order around $\bar{u}\equiv 0$.
\end{proposition}
\begin{proof}
If $D_{\bar{u}}E^{\bar{X},T}$ is not onto, there is a nonzero matrix $Y\in M_{2m}(\R)$ such that 
$$
\bar{X}(T)^* \J Y \in \mathcal{S}(2m)
$$
and
$$
Tr\Bigl(Y^*D_{\bar{u}}E^{\bar{X},T}(v)\Bigr)=0 \qquad v\in L^2\bigl([0,T]; \R^k\bigr). 
$$
By (\ref{raab}), this can be written as
$$
\sum_{i=1}^k\int_0^T v_i(t) Tr(Y^*S(T)S(t)^{-1}B_i\bar{X}(t)) \, dt =0 \quad \forall v \in L^2\bigl([0,T]; \R^k\bigr).
$$
Taking for every $i \in \left\{1,...,k\right\}$,
$$
v_i(t):=Tr(Y^*S(T)S(t)^{-1}B_i\bar{X}(t)) \quad t \in [0,T],
$$
we obtain that 
\begin{equation}\label{0901}
Tr\Bigl(Y^*S(T)S(t)^{-1}B_i\bar{X}(t)\Bigr)=0 \quad \forall t \in [0,T].
\end{equation}
The above equality at $t=\bar{t}$ yields 
$$
Tr\Bigl(Y^*S(T)S(\bar{t})^{-1}B^0_i(\bar{t})\bar{X}(\bar{t})\Bigr)=0.
$$
Using that $\frac{d}{dt}(S(t)^{-1})=-S(t)^{-1}A(t), \dot{\bar{X}}(t)=A(t)\bar{X}(t)$ and differentiating (\ref{0901}) at $t=\bar{t}$ again and again gives
$$
Tr\Bigl(Y^*S(T)S(\bar{t})^{-1}B^j_i(\bar{t})\bar{X}(\bar{t})\Bigr)=0 \quad \forall j\in \N, \forall i \in \left\{1,...,k\right\}.
$$
By (\ref{condsymp}), we have 
$$
\bar{X}(T)^* \J \Bigl(S(T)S(\bar{t})^{-1}B^j_i(\bar{t})\bar{X}(\bar{t})\Bigr) \in \mathcal{S}(2m).
$$
So all the matrices $S(T)S(\bar{t})^{-1}B^j_i(\bar{t})\bar{X}(\bar{t})$ belong to $T_{\bar{X}(T)} \mbox{Sp}(m)$. Since the matrix $S(T)S(\bar{t})^{-1}$ is invertible and (\ref{conditionLIEPROP}) holds, we infer that 
$$
Tr(Y^*H)=0 \quad \forall H \in T_{\bar{X}(T)} \mbox{Sp}(m)
$$
which yields a contradiction.
\end{proof}
As a corollary, we deduce a local controllability property on $\mbox{Sp}(m)$.
\begin{proposition}
\label{PROP2}
Assume that assumptions of Proposition \ref{LIEPROP1} hold. Then, for every $\bar{X} \in \mbox{Sp}(m)$ and $T > 0$, there are $\mu, \nu > 0$, p smooth controls $u^1,\cdots,u^p : [0,T] \rightarrow \R^k$ with $Supp(u^j) \subset (0,T)$ for $j=1,...,p$ and a smooth mapping 
$$
U=(U_1, \cdots, U_p) :  B \Bigl( \bar{X}(T) ,\mu \Bigr) \cap \mbox{Sp}(m)  \longrightarrow B(0,\nu)
$$
with $U\Bigl(\bar{X}(T)\Bigr)=0$ such that for every $X \in B \Bigl( \bar{X}(T) ,\mu \Bigr) \cap \mbox{Sp}(m)$,
$$
E^{\bar{X},T}  \left(  \sum_{j=1}^{p} U_j ( X ) u^j \right) = X.
$$
\end{proposition}
\begin{proof}
Remember that the set of controls $u \in C^{\infty}([0,T],\R^k)$ with $supp (u) \subset (0,T)$  is dense in $L^2([0,T],\R^k)$ and from Proposition \ref{LIEPROP1}, we know that the mapping $E^{\bar{X},T} : L^2 \bigl([0,T]; \R^k\bigr) \rightarrow \mbox{Sp}(m)$ is a smooth submersion at $\bar{u} \equiv 0$. Then there are $p$ smooth controls $u^1,...,u^p: [0,T]\rightarrow \R^k$ with $Supp(u^j) \subset (0,T)$ for $j=1,...,p$ such that
\begin{equation}
\label{3}
Span \left\{DE^{\bar{X},T}(\bar{u})(u^j)\, \vert \, j=1,...,p \right\}=T_{\bar{X}(T)}\mbox{Sp}(m).
\end{equation}
Define $F : \R^p \rightarrow Sp(m)$ by 
$$
F(\lambda):=E^{\bar{X},T}\Bigr(\bar{u}+\sum_{j=1}^p \lambda_j u^j\Bigl) \qquad \forall \lambda=(\lambda_1,...,\lambda_p) \in \R^p.
$$
The function $F$ is well-defined, smooth, and satisfies $F(0)=E^{\bar{X},T}(\bar{u})=\bar{X}(T)$. Its differential at $\lambda=0$ is given by 
$$
DF(0)(\lambda)=\sum_{j=1}^p \lambda_j DE^{\bar{X},T}(\bar{u})(u^j) \quad \forall \lambda \in \R^p,
$$
hence it is invertible By (\ref{3}). By the Inverse Function Theorem, we conclude the proof.
\end{proof}
\begin{remark}
\label{remark}
The radii depend on the size of the datas (see \cite[Theorem B.1.4]{riffordbook}).
\end{remark}
The result below follows easily from Proposition \ref{PROP2}.
\begin{proposition}
\label{PROP3}
Assume that there exists $\bar{t} \in [0,T]$ such that (\ref{conditionLIEPROP}) holds. Then there are $\mu, C > 0$ such that for every $X \in \mbox{Sp}(m)$ with $\left\|X-\bar{X}(T)\right\| < \mu$, there is a $C^{\infty}$ function $u : [0,T] \longrightarrow \R^{\frac{m(m+1)}{2}}$ such that 
$$
Supp(u) \subset (0,T), \quad \left\|u\right\|_{C^k} < C \left\|X-\bar{X}(T)\right\|
$$ 
and
$$
X_u(T)=X.
$$
\end{proposition}

\section{Proof of Theorem 1.1}
 Since $M$ is compact, there exists $\tau > 0$ such that
$$
\gamma \Bigl((1-\tau,1)\Bigr) \cap  \gamma \Bigl([0,1-\tau]\Bigr)= \emptyset,
$$
for every geodesic $\gamma$ arc of $g_0$.
Let $\gamma: \left[0,1\right] \longrightarrow S^gM$ be a geodesic arc of $g_0$ of length $1$ (this can be obtained by scaling).

\vspace*{4cm}

\begin{figure}[!h] 
\centering 
\newgray{grayml}{.9}
\psset{unit=1cm}
\begin{pspicture}(3,3)(6,3)
\psframe[linewidth=2pt,framearc=.3,fillstyle=solid,
fillcolor=yellow](2.5,3.85)(9.65,6.15)

\pscurve[showpoints=false]{-}(8.5,5)(6,5.2)(2,5.2)(-0.5,5)
\pscurve[showpoints=false,linewidth=0.4pt]{-}(7.6,5)(7.8,4.4)(8.7,3.9)(9.6,5)(8.7,6.1)(7.8,5.6)(7.6,5)
\psdots[dotsize=0.7pt](8.5,4.9)(8.5,4.85)(8.5,4.8)(8.5,4.75)(8.5,4.7)(8.5,4.65)(8.5,4.6)(8.5,4.55)(8.5,4.5)(8.5,4.45)(8.5,4.4)(8.5,4.35)(8.5,4.3)(8.5,4.25)
\psline[showpoints=false,linewidth=0.1pt]{<->}(2.5,4.5)(8.5,4.5)
\pscurve[showpoints=false]{-}(8.3,5.5)(6,5.5)(2.5,5.23)
 \rput(9,4.7){$\gamma(1)$}
 \rput(-0.5,4.7){$\gamma(0)$}
 \rput(1.7,4.85){$\gamma(1-\tau)$}
 \rput(5.5,4.7){$\tau > 0$}
 \rput(5.5,5.7){$\tilde{u}=u$}
 \rput(1,5.4){$\tilde{u}=0$}
 \rput(8.6,5.5){$X$}
   \psdots[dotsize=4pt](8.5,5) (-0.5,5)(2.5,5.23) (8.3,5.5)

\end{pspicture}
\caption{Avoiding self-intersection}
\end{figure}
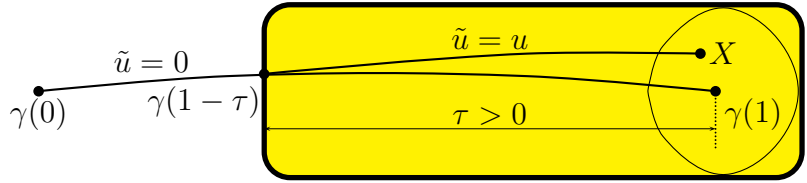
Fix a set of Fermi coordinates $\left\{(t,x)\right\}$ along $\gamma$. The linearized Poincar\'e map $P_{g_{0}}(\gamma)(t)$ satisfies a first order system of the form (see \cite[Section 3]{rr12})
$$
\dot{W}_0(t)=\begin{pmatrix} 0 & I_{n-1} \\ -K(t) & 0 \end{pmatrix} W_0(t) \quad t \in [1-\tau,1],
$$
where $K(t)$ represents the matrix of the sectional curvature of the metric $g_0$. In fact, if $g_0:=(g_0^{kl})_{k,l=0,...,n-1}$, we have for any $i,j=1,...,n-1$,
$$
K(t)_{ij}=-\frac{1}{2}\frac{\partial^2}{\partial x^i \partial x^j}g_0^{00}(t,0).
$$
Let $f: M \longrightarrow \R$ be a $C^2$ function with $f(t,0)=0$ and $\frac{\partial}{\partial x^k}f(t,0)=0$\, $\forall k=1,...,n-1$. Let $u:=(u_{ij})_{i,j=1,...,n-1}$ be the function defined by 
$$
 u_{ij}(t)=-\frac{1}{2}\frac{\partial^2}{\partial x_ix_j}f(t,0), \quad \forall i,j=1,...,n-1 \quad and \quad t \in [1-\tau,1].
 $$
Consider the metric $g_u:=e^f g_0$. The linearized Poincar\'e map $P_{g_{u}}(\gamma)(t)$ is given by 
\begin{equation}\label{1}
\dot{W}_u(t)=\begin{pmatrix} 0 & I_{n-1} \\ -K_{g_u} & 0 \end{pmatrix} W_u(t)  \quad t \in [1-\tau,1],
\end{equation}
where 
\begin{equation}\label{2}
K_{g_u}(t)=K(t) + \sum_{i=1}^{m}u_{ii}(t)E(ii)+\sum_{1 \leq i<j}^{m}u_{ij}(t)E(ij),
\end{equation}
with $E(ij), 1 \leq i\leq j\leq m$ are the symmetric $n-1 \times n-1$ matrices defined by
$$
\mbox{and} \quad \left(E(ij)\right)_{k,l} =   \delta_{ik} \delta_{jl} + \delta_{il} \delta_{jk}  \qquad \forall i, j =1, \ldots,n-1.
$$
Set $m=n-1$, $k:=m(m+1)/2$. The formulas (\ref{1})-(\ref{2}) giving $P_{g_{u}}(\gamma)(t)$ can be viewed as a control system of the form
\begin{eqnarray}\label{syscontrol3}
\dot{X}(t) = A(t) X(t) + \sum_{i\leq j=1}^m u_{ij}(t) \mathcal{E}(ij) X(t),
\end{eqnarray}
where the $2m \times 2m$ matrices $A(t), \mathcal{E}(ij) $ are defined by
$$
A(t) :=  \left( \begin{matrix} 0 & I_m \\ -K(t) & 0 \end{matrix} \right) \qquad \forall t \in [1-\tau,1]
$$
and
$$
\mathcal{E}(ij) :=   \left( \begin{matrix} 0 & 0 \\ E(ij) & 0 \end{matrix} \right).
$$
To avoid eventually self-intersection at $\gamma(1)$, we assume that the support of $u$ is included in $(1-\tau+\delta,1-\delta)$, with $0 < \delta < \tau$.\\
It is clear that if for every final state $X \in \mbox{Sp}(m)$ close to $\gamma(1)$ there is a control $u\in L^2 \bigl([1-\tau,1]; \R^k\bigr)$ which steers $\gamma(1-\tau)$ to $X$ (see figure 1), then the control $\tilde{u}$ defined by
$$
\tilde{u}(t) := \left\{ \begin{array}{cl}0 &  \mbox{ if } t \in [0,1-\tau]\\
u(t) & \mbox{ otherwise}.
\end{array}
\right.
\qquad \mbox{for a.e. } t \in [0,1]. 
$$
steers $\gamma(0)$ to $X$. For sake of simplicity assume from now that $[1-\tau,1]=[0,1]$.\\

Let us first prove the higher-dimensional ($n \geq 3$) Franks' lemma for geodesic flows.\\ 
The Jacobi matrix $K(t)$ is real and symmetric, so it is diagonalisable and there are  $\lambda_1(t),...,\lambda_m(t) \in \R$, $P(t) \in GL_m(\R)$ such that $K(t)=P(t)^{-1}diag\Bigl(\lambda_1(t),...,\lambda_m(t)\Bigr)P(t)$. 
Recall that by hypothesis,
\begin{equation}
\exists \, \,  \bar{t} \in [0,1] \,\, / \,\, \lambda_i(\bar{t}) \ne \lambda_j(\bar{t}), \, \,  \forall i \ne j.
\label{Contreras condition}
\end{equation}
Hence if we change our coordinates, we can suppose that $K(\bar{t})= diag\Bigl(\lambda_1(\bar{t}),...,\lambda_m(\bar{t})\Bigr)$.\\
Since our control system has the form (\ref{syscontrol}), all the results gathered in Section \ref{prelcontrol} apply. Since the $\mathcal{E}(ij)$ do not depend on time, we check easily that the matrices $B_{ij}^0, B_{ij}^1, B_{ij}^2, B_{ij}^3$ associated to our system are given by
$$
\left\{
\begin{array}{l}
B_{ij}^0 (t) = \mathcal{E}(ij) \\
B_{ij}^1(t) = \left[ \mathcal{E}(ij), A(t)\right] \\
B_{ij}^2(t) =   \left[ \left[ \mathcal{E}(ij), A(t)\right], A(t) \right] \\
B_{ij}^3(t) =  \dot{B}_{ij}^2(t)+\left[ \left[ \left[ \mathcal{E}(ij), A(t)\right], A(t) \right], A(t) \right],
\end{array}
\right.
$$
for every $t \in [0,1]$. An easy computation yields for any $i, j =1, \ldots,m$ with $i\leq j$ and any $t\in [0,1]$,
$$
\left[\mathcal{E}(ij), A (t)\right] = \left( \begin{matrix} -E(ij) & 0 \\ 0 & E(ij) \end{matrix} \right),
$$
$$
\left[ \left[ \mathcal{E}(ij), A(t) \right], A(t) \right] =  \left( \begin{matrix} 0 & -2E(ij) \\ -E(ij)K(t)-K(t) E(ij) & 0 \end{matrix} \right),
$$
$$
\left[ \left[ \left[ \mathcal{E}(ij), A(t)\right], A(t) \right], A(t) \right] = \left( \begin{matrix} 3E(ij)K(t)+K(t)E(ij) & 0 \\ 0 & -E(ij)K(t)-3K(t) E(ij) \end{matrix} \right).
$$
We need to show that $S=\mbox{Span} \Bigl\{ B_{ij}^l(\bar{t}) \, \vert \, 1 \leq i \leq j \leq m \,\, and \, \, l=0,1,2,3\Bigr\}$ has dimension $d=2m(2m+1)/2$. For all $1 \leq i \leq j \leq k$ we have\\
\\
$
\left( \begin{matrix} 3E(ij)K(\bar{t})+K(\bar{t})E(ij) & 0 \\ 0 & -E(ij)K(\bar{t})-3K(\bar{t}) E(ij) \end{matrix} \right)= 
$
$$
 2 \left( \begin{matrix} E(ij)K(\bar{t})+K(\bar{t})E(ij) & 0 \\ 0 & -E(ij)K(t)-K(t) E(ij) \end{matrix} \right) + \left( \begin{matrix} \left[E(ij) , K(\bar{t})\right] & 0 \\ 0 & \left[E(ij) , K(\bar{t})\right]) \end{matrix} \right).
$$
Moreover, it holds that
$$
\left( \begin{matrix} 0 & 0 \\ -E(ij)\dot{K}(\bar{t})-\dot{K}(\bar{t})E(ij) & 0 \end{matrix} \right) \in \mbox{Span} \Bigl\{ B_{ij}^0(\bar{t}) \, \vert \, 1 \leq i \leq j \leq m \Bigr\},
$$
and
$$
\left( \begin{matrix} E(ij)K(\bar{t})+K(\bar{t})E(ij) & 0 \\ 0 & -E(ij)K(t)-K(t) E(ij) \end{matrix} \right) \in \mbox{Span} \Bigl\{ B_{ij}^1(\bar{t}) \, \vert \, 1 \leq i \leq j \leq m \Bigr\}.
$$
Let's now compute the $m \times m$ matrices $\left[E(ij) , K(\bar{t})\right]$ for all $1 \leq i < j \leq m$ :
$$
\left[E(ij) , K(\bar{t})\right]:=(c_{rs})_{r,s} \,\, with
\left\{
\begin{array}{l}
c_{rs}=0 \,\,\mbox{if} \,\,(r,s) \ne (i,j) \,\,\mbox{or}\,\, (r,s) \ne (j,i), \\
 c_{ij}= \lambda_j(\bar{t})-\lambda_i(\bar{t}), \\
c_{ji}=\lambda_i(\bar{t})-\lambda_j(\bar{t}).\\
\end{array}
\right.\\
$$
Hence, using the condition (\ref{Contreras condition}) we obtain 
$$
\mbox{span} \left\{ \left( \begin{matrix} \left[E(ij) , K(\bar{t})\right] & 0 \\ 0 & \left[E(ij) , K(\bar{t})\right]) \end{matrix} \right) \, \vert \,  i \leq j  \right\} = \mbox{Span} \left\{ \left( \begin{matrix} F(pq) & 0 \\ 0 & F(pq)) \end{matrix} \right) \, \vert \,  p < q  \right\},\\
$$
where $F(pq)$ is the skew-symmetric matrix defined by
$$
\left( F(pq)\right)_{rs} := \delta_{rp}\delta_{sq} - \delta_{rq}\delta_{sp}.
$$
Therefore we have
$$
S=\mbox{Span} \left\{ B_{ij}^l(\bar{t}), \left( \begin{matrix} F(pq) & 0 \\ 0 & F(pq)) \end{matrix} \right) \, \vert \, 1 \leq i \leq j \leq m \, , \, l=0,1,2\,\, and \,\,1 \leq p < q \leq m \right\}.
$$
This allow us to compute the dimension of $S$. In fact, since the matrices $ \mathcal{E}(ij)$ form a basis of the vector space of symmetric matrices $\mathcal{S}(m)$,  we check easily that the vector space
$$
\mbox{Span} \Bigl\{ \mathcal{E}(ij), \left[ \left[ \mathcal{E}(kl), A(t) \right], A(t) \right] \, \vert \, i, j, k, l\Bigr\}
$$
has dimension $m(m+1)$. It remains to check that the rest spans a space of dimension $d-m(m+1)/2=m^2$. The spaces respectively spanned by
$$
\Bigl\{  \left[ \mathcal{E}(ij), A(t)\right] \, \vert \, i,j \Bigr\}
$$
and
$$
\left\{  \left( \begin{matrix} F(pq) & 0 \\ 0 & F(pq) \end{matrix} \right)\, \vert \, p,q  \right\}
$$
are orthogonal with respect to the scalar product $\mbox{tr} (P^*Q)$. The first has dimension $m(m+1)/2$. It remains to show that the second one has dimension $m(m-1)/2$. The second space is generated by the matrices of the form
$$
   \left( \begin{matrix} F(pq) & 0 \\ 0 & F(pq) \end{matrix} \right)
$$
with $1 \leq p < q \leq m$. Finally, the condition (\ref{conditionLIEPROP}) is satisfied and we conclude easily using Propositions \ref{LIEPROP1}, \ref{PROP3} and a compactness argument (see Remark \ref{remark}).\\

Let us now provide the proof of Franks' lemma for geodesic flows on surfaces.\\
Set $m=1$, the control system (\ref{syscontrol3}) becomes
$$
\dot{X}(t) = A(t) X(t) +  u_{11}(t) \mathcal{E}(11) X(t),
$$

where the $2 \times 2$ matrices $A(t), \mathcal{E}(11)$ are defined by
$$
A(t) :=  \left( \begin{matrix} 0 & 1 \\ -K(t) & 0 \end{matrix} \right) \qquad \forall t \in [0,1]
$$
and
$$
\mathcal{E}(11) :=   \left( \begin{matrix} 0 & 0 \\ 1 & 0 \end{matrix} \right).
$$
Since our control system has the form (\ref{syscontrol}), all the results gathered in Section \ref{prelcontrol} apply. Since the $\mathcal{E}(11)$ do not depend on time, we check easily that the matrices $B_{11}^0, B_{11}^1, B_{11}^2$ associated to our system are given by
$$
\left\{
\begin{array}{l}
B_{11}^0 (t) = \mathcal{E}(11) \\
B_{11}^1(t) = \left[ \mathcal{E}(11), A(t)\right] \\
B_{11}^2(t) =   \left[ \left[ \mathcal{E}(11), A(t)\right], A(t) \right],
\end{array}
\right.
$$
for every $t \in [0,T]$. An easy computation yields for any $t\in [0,T]$,
$$
\left[\mathcal{E}(11), A (t)\right] = \left( \begin{matrix} -1 & 0 \\ 0 & 1 \end{matrix} \right),
$$
$$
\left[ \left[ \mathcal{E}(11), A(t) \right], A(t) \right] =  \left( \begin{matrix} 0 & -2 \\ -2K(t) & 0 \end{matrix} \right).
$$

We check easily that $dim \Bigl( \mbox{Span} \Bigl\{ B_{11}^0(0), \, B_{11}^1(0), B_{11}^2(0) \Bigr\} \Bigr) = 3 = dim \Bigl( T_{I_{2}} \mbox{Sp}(1)\Bigr).$\\
So the condition (\ref{conditionLIEPROP}) is satisfied and the result follows from Propositions \ref{LIEPROP1}, \ref{PROP3} and a compactness argument.

\section{Proof of Proposition 1.2}
Let $\mathcal{F}:=\left\{\xi_1,...,\xi_N \right\}$ be a finite set of geodesic segments that are transverse to $\gamma$, with for every $i=1,...,N$, $\mathcal{\xi}_i$ intersect $\gamma$ at the point $\gamma(t_i)$, where $t_i \in [0,1]$.
\vspace*{7cm}

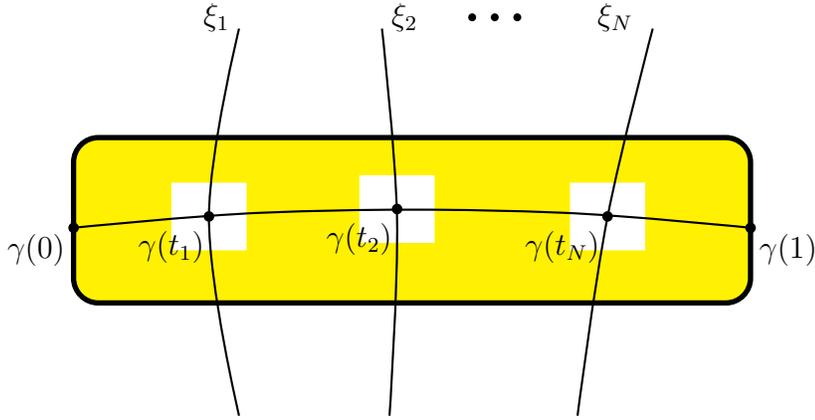
\begin{figure}[!h] 
\centering 
\newgray{grayml}{.9}
\psset{unit=1cm}
\begin{pspicture}(3,2)(6,2)
\psframe[linewidth=2pt,framearc=.3,fillstyle=solid,fillcolor=yellow](-0.5,4)(8.5,6.2)
\psframe*[linecolor=white](.8,4.7)(1.8,5.6)
\psframe*[linecolor=white](3.3,4.8)(4.3,5.7)
\psframe*[linecolor=white](6.1,4.7)(7.1,5.6)
\pscurve[showpoints=false]{-}(8.5,5)(6,5.2)(2,5.2)(-0.5,5)
\pscurve[showpoints=false]{-}(1.7,7.65)(1.3,5.15)(1.7,2.5)
\pscurve[showpoints=false]{-}(3.6,7.65)(3.8,5.15)(3.7,2.5)
\pscurve[showpoints=false]{-}(7.2,7.65)(6.6,5.15)(6.2,2.5)
 \rput(9,4.7){$\gamma(1)$}
 \rput(-1,4.7){$\gamma(0)$}
 \rput(0.8,4.75){$\gamma(t_1)$}
  \rput(3.3,4.85){$\gamma(t_2)$}
   \rput(6,4.75){$\gamma(t_N)$}
\rput(1.4,7.8){$\xi_1$} 
\rput(3.9,7.8){$\xi_2$} 
\rput(6.7,7.8){$\xi_N$} 
 
\psdots[dotsize=4pt](8.5,5) (-0.5,5)(1.3,5.15)(3.8,5.25)(6.6,5.15)
\psdots[dotsize=3pt](4.8,7.8)(5.1,7.8)(5.4,7.8)

\end{pspicture}
\caption{Avoiding a finite number of transverse geodesics}
\end{figure}
From Proposition \ref{PROP2}, we know that there are $p$ smooth controls $u^1,\cdots,u^p : [0,T] \rightarrow \R^k$ with $Supp(u^j) \subset (0,1)$ for $j=1,...,p$, such that the following End-Point mapping (associated to the control system (\ref{syscontrol3}))
$$
\begin{array}{rcl}
E^{I_{2m},1} \, : \, Span\left\{u^1,\cdots,u^p \right\}  & \longrightarrow & \mbox{Sp}(m) \\
\sum _{i=1}^p {\lambda_i u^i} & \longmapsto & X_{I_{2m},\sum _{i=1}^p {\lambda_i u^i}}(1)
\end{array}
$$
is a local diffeomorphism.
Take now $p$ $C^\infty$-functions $\tilde{u}^1,\cdots,\tilde{u}^p : [0,T] \rightarrow \R^k$ such that for every $j=1,...,p$, $Supp(\tilde{u}^j) \subset (0,1)$, $\tilde{u}^j$ vanishes in a neighborhood $\mathcal{N}_i$ of $t_i$ and $\tilde{u}^j$ is a equal to $u^j$ outside of $\mathcal{N}_i$. By $C^1$ regularity of the End-Point mapping $E^{I_{2m},1}$, it holds that the map
$$
\begin{array}{rcl}
\tilde{E}^{I_{2m},1} \, : \, Span\left\{\tilde{u}^1,\cdots,\tilde{u}^p \right\}  & \longrightarrow & \mbox{Sp}(m) \\
\sum _{i=1}^p {\lambda_i \tilde{u}^i} & \longmapsto & X_{I_{2m},\sum _{i=1}^p {\lambda_i \tilde{u}^i}}(1)
\end{array}
$$
remains a local diffeomorphism, which concludes the proof.

%
\newcommand{\noopsort}[1]{}\def\polhk#1{\setbox0=\hbox{#1}{\ooalign{\hidewidth
  \lower1.5ex\hbox{`}\hidewidth\crcr\unhbox0}}} \def\cprime{$'$}
  \def\cydot{\leavevmode\raise.4ex\hbox{.}} \def\cprime{$'$} \def\cprime{$'$}
  \def\cprime{$'$} \def\cprime{$'$} \def\cprime{$'$}
  \def\polhk#1{\setbox0=\hbox{#1}{\ooalign{\hidewidth
  \lower1.5ex\hbox{`}\hidewidth\crcr\unhbox0}}} \def\cprime{$'$}
  \def\cprime{$'$} \def\cprime{$'$} \def\dbar{\leavevmode\hbox to
  0pt{\hskip.2ex \accent"16\hss}d}

\signal
\end{document}